\documentclass[12pt]{amsart}
\usepackage{amsthm,amsfonts,amsmath, amscd, amssymb, hyperref,amssymb}
\usepackage{fullpage}
\usepackage[pdftex]{graphicx}
\usepackage{enumerate}
\usepackage{cite}
\usepackage[usenames,dvipsnames]{color}
\usepackage{subcaption}
\usepackage{graphicx, upgreek}
\usepackage[margin=1in]{geometry}

\makeatletter
\@namedef{subjclassname@2020}{%
  \textup{2020} Mathematics Subject Classification}
\makeatother

\usepackage{datetime,mathtools}
\usepackage{xcolor}
\usepackage{hyperref}
\definecolor{darkgreen}{rgb}{0,0.4,0}
\definecolor{BrickRed}{rgb}{0.65,0.08,0}
\hypersetup{colorlinks=true,linkcolor=blue,citecolor=red,filecolor=BrickRed,urlcolor=darkgreen}

\mathtoolsset{showonlyrefs} 

\newcount\m \newcount\n
\def\hours{\n=\time \divide\n 60
	\m=-\n \multiply\m 60 \advance\m \time
	\twodigits\n:\twodigits\m}
\def\twodigits#1{\ifnum #1<10 0\fi \number#1}
\date{
	\today \quad
	\\
	\textup{2020} Mathematics Subject Classification: \href{https://zbmath.org/classification/?q=cc\%3A11B25}{\path{11B25}}; \href{https://zbmath.org/classification/?q=cc\%3A11N37}{\path{11N37}}; \href{https://zbmath.org/classification/?q=cc\%3A11M50}{\path{11M50}}.
}

\numberwithin{equation}{section}

\title{Variance of the $k$-fold divisor function in arithmetic progressions for individual modulus}

\author[D. T. Nguyen]{David T. Nguyen}
\address{American Institute of Mathematics, 600 E. Brokaw Rd., San Jose, CA 95112, USA.}
\email{dtn@aimath.org}

\address{Current Address: Department of Mathematics and Statistics, Queen's University, Jeffery Hall, 48 University Ave, Kingston, Ontario, K7L-3N6, Canada}
\email{d.nguyen@queensu.ca}

\newtheorem{thm}{Theorem}
\newtheorem{cor}{Corollary}
\newtheorem{remark}{Remark}
\newtheorem{conjecture}{Conjecture}
\newtheorem{lem}{Lemma}

\newtheorem{conjecturex}{Conjecture}
\newtheorem{thmx}{Theorem}

\theoremstyle{definition}

\numberwithin{equation}{section}

\begin{document}

\maketitle

\begin{abstract}
	In this paper, we confirm a smoothed version of a recent conjecture on the variance of the $k$-fold divisor function in arithmetic progressions to individual composite moduli, in a restricted range. In contrast to a previous result of Rodgers and Soundararajan \cite{RodgersSoundararajan2018}, we do not require averaging over the moduli. Our proof adapts a technique of S. Lester \cite{Lester2016} who treated in the same range the variance of the $k$-fold divisor function in the short intervals setting, and is based on a smoothed Vorono\"i summation formula but twisted by multiplicative characters. The use of Dirichlet characters allows us to extend to a wider range from previous result of Kowalski and Ricotta \cite{KowalskiRicotta2014} who used additive characters. Smoothing also permits us to treat all $k$ unconditionally. This result is closely related to moments of Dirichlet $L$-functions.
\end{abstract}

\tableofcontents

\section{Introduction}

\subsection{Background of the problem}

Barban-Davenport-Halberstam-type inequalities involve upper bounding quantities of the shape
\begin{equation} \label{eq:BDH}
	\sum_{d \le D}\
	\sum_{a \pmod d}
	|\Delta(f;X, d, a)|^2,
\end{equation}
where $|\Delta(f;X, d, a)|$ is being squared, with $\Delta(f;X, d, a)$ a certain error term depending on the function $f$, parameters $X,d$, and $a$ coprime to $d$. These types of inequalities originate from the works of Barban \cite{Barban1963, Barban1964} (1963, 1964), Davenport-Halberstam \cite{DavenportHalberstam1966} (1966), and they yield a wilder range for $D$ in terms of $X$ as compared to Bombieri-Vinogradov-type inequalities which bound expressions roughly of the form
	\begin{equation} \label{eq:BV}
		\sum_{d \le D}\
		\max_{a \pmod d}
		|\Delta(f;X, d, a)|.
	\end{equation}
Non-trivial bounds for \eqref{eq:BDH} have many applications in number theory--we give two recent instances below.

1. A version of the Barban-Davenport-Halberstam-type inequality, among with other novel ideas, with the function $f$ replaced by related convolutions over primes, were skillfully used by Zhang \cite[Lemma 10]{Zhang2014} (2014) in his spectacular proof that there are bounded gaps between primes.

2. In their work in 2017, Heath-Brown and Li \cite{HeathBrownLi2017} proved a version of Barban-Davenport-Halberstam inequality in their Corollary 2 as an ingredient to show that the sparse sequence $a^2 + p^4$, where $a$ is a natural number and $p$ is a prime, contains infinitely many primes.

In this paper, we study the asymptotic of a related quantity to \eqref{eq:BDH}. Let $n\ge 1$ and $k\ge 1$ be integers. Let $\tau_k(n)$ denote the {$k$-fold divisor function}
\begin{equation}
	\tau_k(n)=
	\sum_{n_1 n_2 \cdots n_k=n}1,
\end{equation}
where the sum runs over ordered $k$-tuples $(n_1,n_2,\dots,n_k)$ of positive integers for which $n_1 n_2 \cdots n_k=n$, so its Dirichlet generating function is
\begin{equation}
	\zeta(s)^k = \sum_{n=1}^\infty \frac{\tau_k(n)}{n^{s}},\quad
	(\Re s >1).
\end{equation}
Explicitly, if $n= p_1^{\alpha_1} \cdots p_r^{\alpha_r}$ is the prime factorization of $n$, then
\begin{equation}
	\tau_k(n)
	=
	\binom{k+ \alpha_1 - 1}{k-1}
	\cdots
	\binom{k+ \alpha_r - 1}{k-1};
\end{equation}
(see Lemma \ref{lemma:tau}). There is a precise prediction on the asymptotic of the variance
\begin{equation}
	\sum_{\substack{a=1\\ (a,d)=1}}^d
	\left|\sum_{\substack{1\le n\le X\\ n\equiv a (\textrm{mod } d)}} \tau_k(n)
	- \frac{1}{\varphi(d)}
	\sum_{\substack{1\le n\le X\\ (n,d)=1}} \tau_k(n)
	\right|^2,
\end{equation}
(c.f. \eqref{eq:BDH}), as $d,X \to \infty$ at a certain rate, which is the main object of study in this paper. We state this conjecture (in our notation) below.

\begin{conjecturex}[Keating--Rodgers--Roditty-Gershon--Rudnick-Soundararajan]
	\label{conj:3}
	Fix $k\ge 2$. For $X,d\to \infty$ such that $\log X/\log d \to c \in (0,k)$, we have
	\begin{equation} \label{eq:2238}
			\sum_{\substack{a=1\\ (a,d)=1}}^d
			\left|\sum_{\substack{1\le n\le X\\ n\equiv a (\textrm{mod } d)}} \tau_k(n)
			- \frac{1}{\varphi(d)}
			\sum_{\substack{1\le n\le X\\ (n,d)=1}} \tau_k(n)
			\right|^2
			\sim a_k(d) \gamma_k(c) X (\log d)^{k^2-1},
		\end{equation}
	where $a_k(d)$ is the arithmetic constant
	\begin{equation} \label{eq:akd1}
			a_k(d) = \lim_{s\to 1^+}
			(s-1)^{k^2}
			\sum_{\substack{n=1\\ (n,d)=1}}^\infty
			\frac{\tau_k(n)^2}{n^s},
		\end{equation}	
	 and $\gamma_k(c)$ is a piecewise polynomial of degree $k^2 - 1$ defined by
	 \begin{equation} \label{eq:gammakcomplicatedexpression}
		 	\gamma_k(c) = \frac{1}{k! G(k+1)^2}
		 	\int_{[0,1]^k}
		 	\delta_c (w_1 + \cdots w_k)
		 	\Delta(w)^2 d^kw,
		 \end{equation}
	 where $\delta_c(x) = \delta(x-c)$ is a Dirac delta function centered at $c$, $\Delta(w) = \prod_{i<j} (w_i - w_j)$ is a Vandermonde determinant, and $G$ is the Barnes $G$-function, so that in particular $G(k + 1) =
	 (k - 1)! (k - 2)! \cdots 1!$.
\end{conjecturex}


This conjecture originates from the work of Keating, Rodgers, Roditty-Gershon, and Rudnick (or, KRRR, for short) in 2018 \cite[Conjecture 3.3]{KR^32018} for prime moduli, with general form for composite moduli put forward by Rodgers and Soundararajan in \cite[Conjecture 1]{RodgersSoundararajan2018} later that year. We give two reasons why this conjecture is relevant.

Firstly, Conjecture \ref{conj:3} is closely related to the problem of moments of Dirichlet $L$-functions \cite{ConreyGonnek2002} and correlations of divisor sums \cite{ConreyKeatingI,ConreyKeatingII,ConreyKeatingIII,ConreyKeatingIV,ConreyKeatingV}. For instance, let $g_k$ denote the geometric factor in the leading term asymptotic of the 2$k$-th moment of the Riemann zeta function on the critical line:
\begin{equation} \label{eq:moment}
	\int_0^T \left| \zeta\left(\frac{1}{2} + it\right) \right|^{2k} dt
	\sim a_k g_k T\frac{(\log T)^{k^2}}{k^2!}, (T \to \infty),
\end{equation}
where
\begin{equation}
	a_k = 
	\prod_p 
	\left(1-\frac{1}{p}\right)^{(k-1)^2}
	\left(
	1 + \frac{\binom{k-1}{1}^2}{p}
	+ \frac{\binom{k-1}{2}^2}{p^2}
	+ \cdots
	\right).
\end{equation}
The asymptotic \eqref{eq:moment} is currently only known, in the case $k$ is a natural number, for $k=1$ and $2$. Then, the piece-wise polynomial $\gamma_k(c)$ in the asymptotic \eqref{eq:2238} is related to the constant $g_k$ in the moment conjecture \eqref{eq:moment} by the conjectural relation
\begin{equation} \label{eq:243}
	k^2! \int_0^k \gamma_k(c) dc
	= g_k,\quad
	(k \ge 1).
\end{equation}
The connection \eqref{eq:243} thus provides an alternative route towards \eqref{eq:moment} via \eqref{eq:2238} for $k\ge 3$. Explicit expressions for $\gamma_k(c)$ for $1\le k \le 6$ are given in \cite[Tables I and II]{Basoretal}. For example, when $k=3$,
\begin{equation}
	\gamma_3(c)
	= \begin{cases}
		\frac{1}{8!} c^8, & \text{ if } 0 \le c < 1,
		\\
		\frac{1}{8!} \left(
			-2c^8 + 24c^7 - 252c^6 + 1512 c^5 - 4830c^4
			\right.
			\\
			\left.
			\quad \quad
			+ 8568c^3 - 8484c^2 + 4392c - 927
		\right), & \text{ if } 1 \le c < 2,
		\\
		\frac{1}{8!} (3-c)^8, & \text{ if } 2 \le c \le 3,
	\end{cases}
\end{equation}
and a simple integration yields
\begin{equation}
	9! \int_0^3 \gamma_3(c) dc
	= 42,
\end{equation}
which is equal to the conjectural value
\begin{equation}
	g_k = \frac{(k^2)!}{ 1 \cdot 2^2 \cdot k^k \cdot (k+1)^{k-1} \cdot (2k+1)}
\end{equation}
when $k=3$. The plot of the function $9! \gamma_3(c)$ is shown in Figure \ref{fig:gamma3}.

\begin{figure}
	\caption{Plot of the piece-wise polynomial $9! \gamma_3(c)$ versus $c$.}
	\label{fig:gamma3}
	\includegraphics{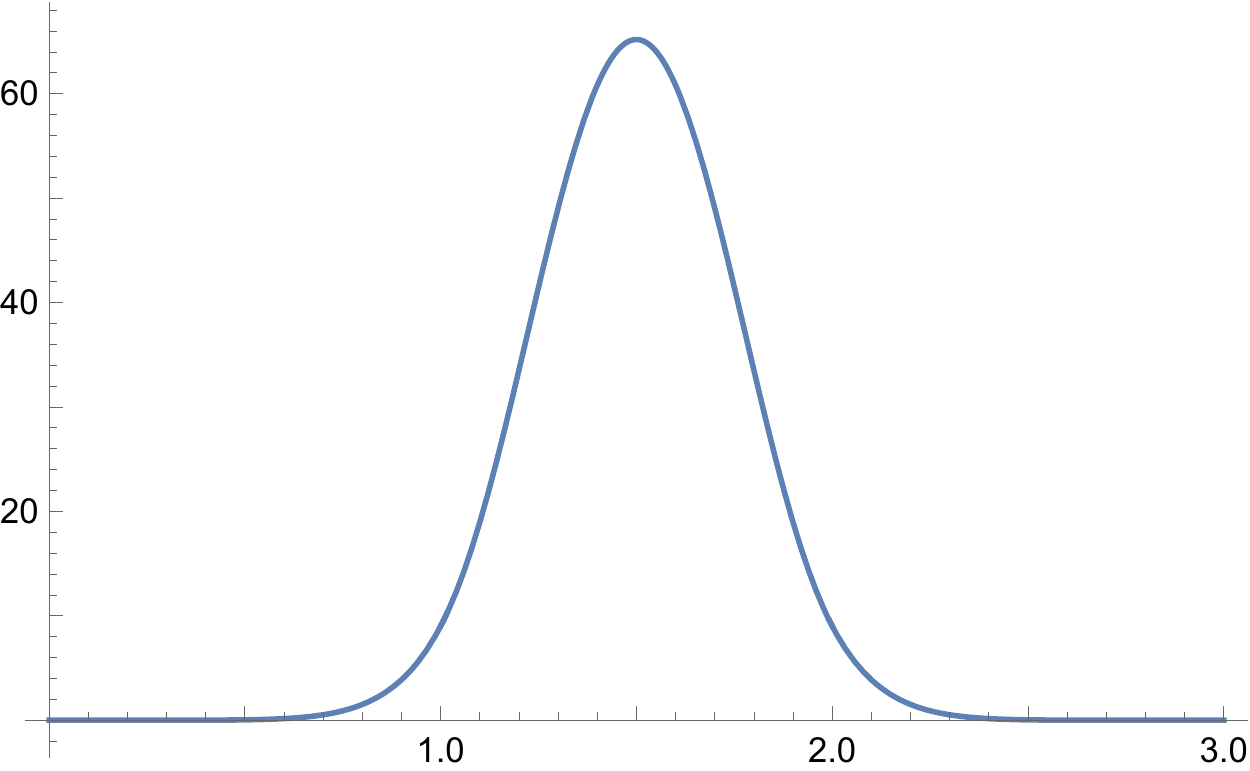}
\end{figure}

Secondly, Conjecture \ref{conj:3} suggests that, on average, the $k$-fold divisor function has a ``level of distribution" up to level $1-\epsilon$ in arithmetic progressions for any $k\ge 2$, that is,
\begin{equation} \label{eq:level}
	\Delta(\tau_k; X, d,a) := 
	\sum_{\substack{1\le n\le X\\ n\equiv a (\textrm{mod } d)}} \tau_k(n)
	- \frac{1}{\varphi(d)}
	\sum_{\substack{1\le n\le X\\ (n,d)=1}} \tau_k(n)
	\ll \frac{X^{1-\epsilon}}{\varphi(d)},\
	((a,d) = 1),
\end{equation}
holds for all $d \ll X^{1-\epsilon}$. Indeed, if \eqref{eq:2238} held for all $c \in (0,k)$, then, we have roughly that
\begin{equation}
	d |\Delta(\tau_k; X, d,a)|^2 = X^{1+\epsilon},
\end{equation}
so that $|\Delta(\tau_k; X, d,a)|$ is roughly of size, with $d = X^{1/c}$,
\begin{equation}
	\sqrt{\frac{X^{1+\epsilon}}{d}}
	= X^{\frac{c-1}{2c} + \frac{1}{2\epsilon}}.
\end{equation}
Thus, the bound \eqref{eq:level} holds if the above
\begin{equation}
	X^{\frac{c-1}{2c} + \frac{1}{2\epsilon}}
	\ll X^{\frac{c-1}{c} - \epsilon},
\end{equation}
or
\begin{equation}
	\frac{1}{c} \le 1-\frac{3 \epsilon}{2},
\end{equation}
and this happens when
\begin{equation}
	d \ll X^{1-\frac{3 \epsilon}{2}}.
\end{equation}
Known levels of distribution for $\tau_k$ are summarized in Table 1 in \cite[p. 33]{NguyenDivisorFunction}. We next briefly survey what is known about Conjecture \ref{conj:3}.

In \cite{RodgersSoundararajan2018}, Rodgers and Soundararajan confirmed an averaged version of Conjecture \ref{conj:3} in a restricted range of $c$, and over smooth cutoffs. Harper and Soundararajan in \cite{HarperSound} obtained a lower bound of the right order of magnitude for the average of this variance \eqref{eq:2238}. Nguyen \cite[Theorem 3]{NguyenDivisorFunction}, by using the large sieve inequality, obtained a matching upper bound of the same order of magnitude for this  variance when averaged over the moduli $d$.

More is known about the variance \eqref{eq:2238} when the modulus $d$ is prime. For instance, when $k\ge 3$ and $d$ is prime, Kowalski and Ricotta in 2014 computed, as one of many results in \cite[Theorem C, pg. 1235]{KowalskiRicotta2014} a smoothed version of the variance \eqref{eq:2238}, for
\begin{equation} \label{eq:205}
	c \in \left(k-\frac{1}{2} , k \right), \quad (k\ge 3).
\end{equation}
Their proof uses deep equidistribution result of
products of hyper-Kloosterman sums and is based on a Vorono\"i summation formula for $GL(N)$ twisted by additive characters. This theme of using additive characters to study moments of arithmetic sequences in progressions has also been pursued by de la Bret\`eche and Fiorilli in \cite{delaBretecheFiorilli2000}.

In KRRR (\cite{KR^32018}), an analogous conjecture to the variance of divisor sums in arithmetic progressions was also made for divisor sums in short intervals. Lester \cite{Lester2016} in 2016 succeeded in evaluating this variance in the short intervals setting, for $c\in (k-1,k)$, unconditionally for $k=3$ and conditionally on the Lindel\"of hypothesis for all $k>3$. His proof combines two previous methods of Selberg and Jutila, and is based in part on a summation formula for coefficients of $L$-functions developed in \cite{FriedlanderIwaniec2005} by Friedlander and Iwaniec in 2005. We now state our main result.

\subsection{Statements of results}

Let $g(n)$ be an arithmetic function, and $w(y)$ a smooth weight function compactly supported on the positive real numbers. For $(a,d)=1$, define
\begin{equation}
	\Delta_w(g;X,d,a)
	:=
	\sum_{\substack{n\equiv a (d)}}
	g(n) w\left(\frac{n}{X}\right)
	- \frac{1}{\varphi(d)}
	\sum_{\substack{(n,d)=1}}
	g(n) w\left(\frac{n}{X}\right).
\end{equation}
In this paper, we confirm a smoothed version of the variance \eqref{eq:2238} of $\tau_k(n)$ in arithmetic progressions to any composite modulus $d$, for $c\in (k-1,k)$, for any fixed $k \ge 3$. Our procedure follows, with some modifications, that of Lester in his work \cite{Lester2016}. This is

\begin{cor}[Main result] \label{cor:2}
	Fix $k\ge 3$. Let $w(y)$ be a smooth function with support in $[1,2]$ such that
	\begin{equation} \label{eq:weightwNormalization}
		\int w(y)^2 dy =1.
	\end{equation}
	Then, uniformly in $c$ for all
	\begin{equation} \label{eq:star}
		c \in (k-1,k),
	\end{equation}
	we have the asymptotic
	\begin{align} \label{eq:cor:2}
		\sum_{\substack{1\le a\le d\\ (a,d)=1}}
		\left|
		\Delta_w(\tau_k;X,d,a)
		\right|^2
		\sim
		a_k(d) \gamma_k(c) X (\log d)^{k^2-1},
	\end{align}
	as $d\to \infty$, where $a_k(d)$ is given as in \eqref{eq:akd1},
	\begin{equation} \label{eq:X}
		X = d^c,
	\end{equation}
	and
	\begin{equation} \label{eq:gammakcsimple}
		\gamma_k(c) = \frac{1}{(k^2-1)!} (k-c)^{k^2-1},\ (k-1 \le c < k).
	\end{equation}
\end{cor}

\begin{remark}
	By Theorem 1.5 \cite[p. 173]{KR^32018}, Lemma 4.1 \cite[p. 180]{KR^32018}, and the results in Section 4.4.3 \cite[pg. 194]{KR^32018} of KRRR, the complicated expression for $\gamma_k(c)$ given in \eqref{eq:gammakcomplicatedexpression} agrees with the expression \eqref{eq:gammakcsimple}, when $k-1 \le c < k$.
\end{remark}

\begin{remark}
	Because of the normalization \eqref{eq:weightwNormalization} on the weight $w(y)$, the right side of \eqref{eq:cor:2} does not depend on $w$; see Lemma \ref{lemma:w} below.
\end{remark}


The leading order main term of \eqref{eq:cor:2} is extracted from the full main term \eqref{eq:thm1} below, given as a contour integral with a power saving error term. All other lower order main terms of the smoothed variance \eqref{eq:cor:2} can, in principal, be written down from \eqref{eq:thm1}.

\begin{thm}[Full main term with power-saving error term] \label{theorem:1}
	Fix parameters 
	\begin{equation} \label{eq:k&delta}
		k\ge 3,\ \delta >0,
	\end{equation}
	and
	\begin{equation} \label{eq:c}
		c \in [k-1+\delta, k-\delta].
	\end{equation}
	For $d\ge 1$, put 
	\begin{equation}
		X = d^c.
	\end{equation}
	Let $w(y)$ be a smooth function with support in $[1,2]$ such that $\int w(y)^2 dy =1$.
	Then, we have the asymptotic equality, as $d\to \infty$,
	\begin{align} \label{eq:thm1}
		&\sum_{\substack{1\le a\le d\\ (a,d)=1}}
		\left|
		\Delta_w(\tau_k;X,d,a)
		\right|^2
		= X M_{k,c}(d) + O_{\epsilon,\delta, k} \left(
		X^{1+\epsilon} d^{-\frac{\delta}{3k + 2}}
		\right),
	\end{align}
	for some $\epsilon>0$, where 
	\begin{equation}
		M_{k,c}(d)
		=
		\frac{1}{\varphi(d)}
		\sum_{\substack{d=qr\\ q\ge Q}}
		\frac{1}{2 \pi i} \int_{(2)}
		\left(\frac{(q/ \pi)^k}{X} \right)^{s-1}
		\sum_{\substack{n=1\\ (n,r)=1}}^\infty
		\frac{\tau_k(n)^2}{n^s} 
		\displaystyle\sideset{}{^*}\sum_{\chi_1 \pmod q}
		\mathcal{M}[|f_{r, \chi_1}|^2](s) ds,
	\end{equation}
	with $Q$ given in \eqref{eq:Q}, $\mathcal{M}[|f_{r, \chi_1}|^2]$ denoting the Mellin transform of $|f_{r, \chi_1}(\cdot)|^2$, and $f_{r, \chi_1}(\cdot)$ a complicated inverse Mellin transform of the Gamma factors and $w(y)$, given explicitly in \eqref{eq:128}. The $^*$ on the summation over $\chi_1 (\bmod\ q)$ restricts to summing over primitive characters $\chi_1$ modulo $q$. The implied constant in the error term of \eqref{eq:thm1} is effective, dependent on $\epsilon, \delta, k$, but uniform in $c$ for all $c$ in the range \eqref{eq:c}.
\end{thm}

\begin{remark}
	Because of the $\epsilon$ in the error term of \eqref{eq:thm1}, our method of proof fails to handle the endpoint $c=k-1$, corresponding to $\delta=0$. More specifically, this restriction came from trivially bounding the off-diagonal terms; see the bound \eqref{eq:off-diag}. Thus, when $c\le k-1$, the off-diagonal terms start contributing to the main term and can no longer be ignored.
\end{remark}

\begin{remark}
	It may be possible, by using the asymptotic large sieve of Conrey-Iwaniec-Soundararajan \cite{ConreyIwaniecSound2012} together with the functional equation, to handle the off-diagonal terms and extend the range for $c$ to a slightly wider range
	\begin{equation} \label{eq:1145b}
		c \in \left[
		k-1 + \delta
		-\frac{2}{k},
		k-\delta
		\right],
		\quad
		(\delta > 0),
	\end{equation}
	for the smoothed averaged variance
	\begin{equation} \label{eq:1145}
		\sum_{d} W\left(\frac{d}{D}\right)
		\sum_{\substack{1\le a\le d\\ (a,d)=1}}
		\left|
		\Delta_w(\tau_k;X,d,a)
		\right|^2
	\end{equation}
	with
	\begin{equation}
		D = X^{1/c}.
	\end{equation}
	This average \eqref{eq:1145} in $d$ could be quite short, in contrast to the typical Barban-Davenport-Halberstam-type results (c.f. \eqref{eq:BDH}). We hope to return to this question in a future article.
\end{remark}

\begin{remark}
	It remains an interesting challenge to remove the smooth weight $w(n/X)$ in \eqref{eq:thm1}.
\end{remark}

We believe that the asymptotic for the smoothed variance \eqref{eq:cor:2} should hold for all $c \in (0,k)$, with the expression for the piece-wise polynomial $\gamma_k(c)$ given in \eqref{eq:gammakcomplicatedexpression}. The asymptotic \eqref{eq:cor:2} for $0<c<1$ was essentially established in \cite[Lemma 2, pg. 13]{RodgersSoundararajan2018}. This, together with Corollary \ref{cor:2} above, therefore suggests the following conjecture over smooth cutoffs for the variance of $\tau_k(n)$ in AP's:

\begin{conjecture} \label{conj:1}
	Let $w(y)$ be a smooth function supported in $[1,2]$ with
	\begin{equation}
		\int w(y)^2 dy = 1,
	\end{equation}
	and
	\begin{equation}
		\mathcal{M}[w](\sigma + it) \ll_\ell \frac{1}{1+ |t|^\ell}
	\end{equation}
	uniformly for all $|\sigma| \le A$ for any fixed positive $A>0$, for all positive integers $\ell$, where $\mathcal{M}[w]$ denotes the Mellin transform of $w$ defined in Section \ref{sec:Notations}.
	Then, for $X, d \to \infty$ such that $\frac{\log X}{\log d} \to c \in (0,k)$, we have
	\begin{equation}
		\sum_{\substack{1\le a\le d\\ (a,d)=1}}
		\left|
		\Delta_w(\tau_k;X,d,a)
		\right|^2
		\sim
		a_k(d) \gamma_k(c) X (\log d)^{k^2-1},
	\end{equation}
	where the arithmetic constant $a_k(d)$ and the piecewise polynomial $\gamma_k(c)$ are given in \eqref{eq:akd1} and \eqref{eq:gammakcomplicatedexpression}, respectively.
\end{conjecture}

We will prove Theorem \ref{theorem:1} first, then derive Corollary \ref{cor:2} from this.

Below is the organization of the paper. The key ideas of the proofs are outlined in Section \ref{sec:outlineOfTheProof}. Section \ref{sec:Notations} lists notations, parameters, and conventions used in this paper, which largely follows those in \cite{NguyenDivisorFunction}. Preparatory lemmas are collected in Section \ref{sec:lemmas}, and the proofs are given in Sections \ref{sec:ProofOfTheorem} and \ref{sec:ProofofCorollary}.

\section{Outline of the proofs} \label{sec:outlineOfTheProof}

We briefly explain here the wider range \eqref{eq:star} for $c$ in our result as compared to \eqref{eq:205}, and how we handle the off-diagonal and error terms. For ease of exposition, suppose $\chi$ is a primitive Dirichlet character modulo a prime $d=p$. By orthogonality of Dirichlet characters, the variance $$\sum_{\substack{1\le a\le d\\ (a,d)=1}}
\left|
\Delta_w(\tau_k;X,d,a)
\right|^2$$ is transformed to sums of the primitive character sum
\begin{equation} \label{eq:1148}
	\sum_{n\le X} \tau_k(n) \chi(n)
\end{equation}
of length $X$ (see Lemma \ref{lem:non-principalCharacters}). By Perron's formula (or Mellin inversion, in the case with smooth weights), the sum \eqref{eq:1148} is related to the quantity
\begin{equation} \label{eq:302}
	X^s L^k(s,\chi)
\end{equation}
appearing in the integrand of such formula. By a change of variables $s \to 1-s$ and applying the functional equation to $L(1-s,\chi)$, $\chi$ primitive, the quantity \eqref{eq:302} becomes
\begin{equation} \label{eq:312}
	X^{1-s} \gamma^k(s,\chi) L^k(s,\overline{\chi}),
\end{equation}
which is roughly, by standard estimates of the gamma factors,
\begin{equation} \label{eq:304}
	X^{1-s} {(p^k)}^{s-\frac{1}{2}} L^k(s,\overline{\chi})
	=
	\left(\frac{X}{Y}\right)^{1/2}
	Y^s L^k(s,\overline{\chi}),
\end{equation}
where
\begin{equation}
	XY = p^k.
\end{equation}
Thus, by Perron's formula again, the quantity $Y^s L^k(s,\overline{\chi})$ on the right side of the above is related to the dual sum
\begin{equation}
	\sum_{n \le Y} \tau_k(n) \overline{\chi}(n)
\end{equation}
of length $Y$. The upshot is, while $X = p^c$ is long when $c \in (k-1,k)$,
\begin{equation}
	Y = \frac{p^k}{X}
	= p^{k-c} < p
\end{equation}
is short for $c$ in this range. This implies that, if we had $n\equiv m (\bmod \ p)$ and $n,m \le Y$, which is $< p$, then $n$ and $m$ must be equal and there are no contributions from the off-diagonal terms $n \neq m$. 

More over, since the modulus is prime, all primitive characters modulo $p$ are non-principal, and we have, by orthogonality of all characters, that
\begin{equation}
	\frac{1}{\varphi(p)} \displaystyle\sideset{}{^*}\sum_{\chi (\text{mod } p)}
	\chi(n) \overline{\chi}(m)
	= 1 - \frac{1}{\varphi(p)}
	= 1 + O(p^{-1}).
\end{equation}
Hence, roughly speaking, the saving $p^{-1}$ in the error term of the above allows us to take $c$ to be as small as $k-1$, where the $-1$ corresponds to the $-1$ in the exponent of $p$. Moreover, direct use of summation formulas from \cite{FriedlanderIwaniec2005} when applied to the variance \eqref{eq:2238} yields an error term that is larger than the main term--this is because the absolute value squared in \eqref{eq:2238} requires better than square-root cancellation in the error term, say, $E$ in order for $|E|^2 \ll X^{1-\epsilon}$. We circumvent this by introducing smooth weights in \eqref{eq:2238}, and combine this with the methods of Friedlander and Iwaniec in \cite{FriedlanderIwaniec2005} in order to get better control on the error terms. These two ingredients in part lead to the full asymptotic equality, with a power-saving error term, for this smoothed version of the variance of $\tau_k(n)$ in arithmetic progressions unconditionally for all $k\ge 3$, in a restricted range. 

Since the proofs are quite lengthy and technical in details, we shall divide the proofs into smaller steps and outline in this section the key ideas in each step.

\textbf{Outline of the proof of Theorem \ref{theorem:1}}

The main idea in the proof of \eqref{eq:thm1} is to develop a smoothed version of Vorono\"i summation twisted by multiplicative characters for the character sum \eqref{eq:1148}.
We break the proof of Theorem \ref{theorem:1} into four steps.

\nameref{subsection:step1}. We first rewrite the variance \eqref{eq:thm1} as a sum over primitive characters and apply the functional equation to $L(1-s,\chi)$. We show that contributions from characters $\chi$ with modulus $q<Q$ are negligible, where the parameter $Q$ is suitably chosen. 

\nameref{subsection:step2}. Assume $q \ge Q$. We introduce parameter $N$ and show that contributions from terms with $n>N$ in $\sum_{\substack{(n,r)=1}}
\frac{\tau_k(n) \overline{\chi_1}(n)}{n^s}$ are acceptable by pushing the line of integration far to the right.

\nameref{subsection:step3}. Assume $q \ge Q$ and $n\le N$. We estimate the off-diagonal contributions trivially by using Lemma \ref{lemma:tauBound} and show they do not exceed the main term. 
It is here that the lower bound for the range of $c$ \eqref{eq:c} in Theorem \ref{theorem:1} is used crucially.

\nameref{subsection:step4}. Assume $q \ge Q$ and $n\le N$. We apply Mellin inversion to the finite sum over $n\le N$ and extract the contribution from the diagonal terms as a contour integral.

\textbf{Outline of the proof of Corollary \ref{cor:2}} 

We extract the leading order term in the full main term from Theorem \ref{theorem:1} and also need to show that the other main terms are indeed of lower order than the leading main term. A certain arithmetic function arises which is then evaluated to match the arithmetic factor $a_k(d)$. Below are the steps in the proof of Corollary \ref{cor:2}.

\nameref{subsection:stepa}. We write the contour integral from \eqref{eq:thm1} as a residue plus an error term. This is fairly straightforward.

\nameref{subsection:stepb}. We show that contribution from the error term in \textbf{Step a} is small. It is here that the condition $q\ge Q$ is needed.

\nameref{subsection:stepd}. We write the residue from \textbf{Step a} as a polynomial. We then evaluate the contribution from the leading term of the residue from \textbf{Step a} and verify that the leading constants match.

\nameref{subsection:stepc}. We show that contributions from lower order terms are of order of magnitude smaller than the leading order term.

\section{Notations, parameters, and conventions}
\label{sec:Notations}

$i,j,\ell$--non-negative integers.

$p, p_i, p_j$--are prime numbers.

$d,a,n,m, k, q,r,s, \alpha_i, \beta_j, Q, T, N$--positive integers.

$X$--a large real number.

$\Lambda(n)$--the von Mangoldt function.

$\tau_k(n)$--the $k$-fold divisor function; $\tau_2(n) = \tau(n)$, the usual divisor function.

$\varphi(n)$--the Euler's totient function.

$\varphi^*(q)$--the number of primitive characters modulo $q$.

$\mu(n)$--the M\"obius function.

$s=\sigma+ it$	

$\Gamma(s)$-the Gamma function.

$\int_{(\sigma)}$--means $\int_{\sigma-i \infty}^{\sigma + i \infty}$.

$\mathcal{M}[f](s)$--the Mellin transform of a suitable function $f$, i.e,
\begin{equation}
	\mathcal{M}[f](s)
	=
	\int_0^\infty
	f(x) x^{s-1} dx.
\end{equation}

$\mathcal{M}^{-1}[F](x)$--the inverse Mellin transform of a suitable function $F(s)$, i.e,
\begin{equation}
	\mathcal{M}^{-1}[F](x)
	= 
	\frac{1}{2 \pi i}
	\int_{(\sigma)}
	F(s) x^{-s} ds,\ (\sigma >1).
\end{equation}

$\chi(n)$--a Dirichlet character.

$\overline{\chi}(n)$--means the complex conjugate of $\chi(n)$.

$\tau(\chi)$--the Gauss sum of the character $\chi$.

$m\equiv a (d)$ means $m\equiv a (\bmod\ d).$

$\alpha \star \beta$ means the Dirichlet convolution of the arithmetic functions $\alpha$ and $\beta$, i.e.,
\begin{equation}
	(\alpha \star \beta)(d)
	= \sum_{d=qr} \alpha(q) \beta(r).
\end{equation}

$\displaystyle\sideset{}{'}\sum_{\chi (\text{mod } d)}$--means a summation over nonprincipal characters $\chi (\bmod\ d).$

$\displaystyle\sideset{}{^*}\sum_{\chi (\text{mod } d)}$--means a summation over primitive characters $\chi (\bmod\ d).$

$\displaystyle\sum_{b (\text{mod } d)}$--means
$\displaystyle\sum_{b=1}^d$.

$\displaystyle\sideset{}{^*}\sum_{b (\text{mod } d)}$--means
$ \displaystyle\sum_{\substack{b=1\\ (b,d)=1}}^d$.

$\epsilon$--any sufficiently small, positive constant, not necessarily the same in each occurrence.

$\delta$--a fixed small positive number.

$A$--any sufficiently large, positive constant, not necessarily the same in each occurrence.

$B$--some large positive constant, not necessarily the same in each occurrence.

\begin{table}[h!]
	\caption{Table of parameters and their first appearance.}
	\label{table:3}
	\begin{tabular}{ l c}
		\hline \hline
		Parameters& First appearance\\
		\hline \hline
		$k \ge 3$& \\
		$\delta >0$& \eqref{eq:k&delta}\\
		$c \in [k-1 + \delta, k - \delta]$& \eqref{eq:c}\\
		$X = d^c$& \eqref{eq:X}\\
		$Q = d^{\frac{3c + 2}{3k + 2}}$& \eqref{eq:Q}\\
		\hypertarget{parameter:N}{$N = d^{k - c + \frac{\delta}{2}}$}
		& \eqref{eq:N}\\
		$P = (\log d)^2$& \eqref{eq:P}\\
		\hline \hline
	\end{tabular}
\end{table}

We follow standard notations and write $f(X)=O(g(X))$ or $f(X) \ll g(X)$ to mean that $|f(X)| \le Cg(X)$ for some fixed constant $C$, and $f(X) = o(g(X))$ if $|f(X)|\le c(X) g(X)$ for some function $c(X)$ that goes to zero as $X$ goes to infinity.

\bigskip
We begin the proofs with some preliminary lemmas.

\section{Preparatory Lemmas}
\label{sec:lemmas}


\begin{lem}\label{lemma:tauBound}
	For any $\epsilon>0$, we have
	\begin{equation} \label{eq:tauBound}
		\tau_k(n)
		\ll_{\epsilon,k} n^\epsilon.
	\end{equation}
\end{lem}
\begin{proof}
	See, e.g, \cite[(1.81), p. 23]{IwaniecKowalski}.
\end{proof}

\begin{lem} \label{lem:non-principalCharacters}
	We have
	\begin{align} \label{eq:1025}
		\sum_{\substack{1\le a\le d\\ (a,d)=1}}
		\left|
		\Delta_w(\tau_k;X,d,a)
		\right|^2
		= \frac{1}{\varphi(d)}
		\displaystyle\sideset{}{^\prime} \sum_{\chi \pmod d}
		\left| \sum_n \tau_k(n) \chi(n) w\left(\frac{n}{X}\right) \right|^2.
	\end{align}
\end{lem}
\begin{proof}
	By the orthogonality relation
	\begin{equation}
		\frac{1}{\varphi(d)}
		\sum_{\chi \pmod d}
		\chi(n) \overline{\chi}(a)
		= \begin{cases}
			1,& \text{if } n\equiv a (d),\\
			0,& \text{otherwise},
		\end{cases}
	\end{equation}
	we have
	\begin{equation}
		\Delta_w(\tau_k; X, d,a)
		= 
		\frac{1}{\varphi(d)}
		\displaystyle\sideset{}{^\prime}\sum_{\chi \pmod d}
		\overline{\chi}(a)
		\sum_{\substack{n}}
		\tau_k(n) 
		\chi(n)
		w\left(\frac{n}{X}\right).
	\end{equation}
	Thus, the left side of \eqref{eq:1025} can be written as
	\begin{align} \label{eq:1001}
		\sum_{\substack{1\le a\le d\\ (a,d)=1}}
		&\frac{1}{\varphi(d)^2}
		\displaystyle\sideset{}{^\prime}\sum_{\chi_1 \pmod d}
		\overline{\chi_1}(a)
		\sum_{\substack{n}}
		\tau_k(n) 
		\chi_1(n)
		w\left(\frac{n}{X}\right)
		\\&\qquad \times
		\displaystyle\sideset{}{^\prime}\sum_{\chi_2 \pmod d}
		\chi_2 (a)
		\sum_{\substack{m}}
		\tau_k(m) 
		\overline{\chi_2}(m)
		w\left(\frac{m}{X}\right)
		\\&= \frac{1}{\varphi(d)^2}
		\displaystyle\sideset{}{^\prime}\sum_{\chi_1, \chi_2 \pmod d}
		\sum_{\substack{n,m}}
		\tau_k(n) 
		\chi_1(n)
		w\left(\frac{n}{X}\right)
		\\& \qquad \times
		\tau_k(m) 
		\overline{\chi_2}(m)
		w\left(\frac{m}{X}\right)		
		\sum_{\substack{1\le a\le d\\ (a,d)=1}}
		\overline{\chi_1}(a)
		\chi_2(a).
	\end{align}
	By the relation
	\begin{equation}
		\sum_{\substack{1\le a\le d\\ (a,d)=1}}
		\overline{\chi_1}(a)
		\chi_2(a)
		= \begin{cases}
			\varphi(d),& \text{if } \chi_1 = \chi_2,\\
			0,& \text{otherwise},
		\end{cases}
	\end{equation}
	the right side of \eqref{eq:1001} is equal to
	\begin{equation}
		\frac{1}{\varphi(d)}
		\displaystyle\sideset{}{^\prime}\sum_{\chi \pmod d}
		\sum_{\substack{n,m}}
		\tau_k(n) 
		\chi(n)
		w\left(\frac{n}{X}\right)
		\tau_k(m) 
		\overline{\chi}(m)
		w\left(\frac{m}{X}\right)	
	\end{equation}
	which is equal to the right side of \eqref{eq:1025}.
\end{proof}

We next reduce non-principal to primitive characters.

\begin{lem} \label{lem:primitiveCharacters}
	We have
	\begin{equation} \label{eq:1125b}
		\displaystyle\sideset{}{^\prime}\sum_{\chi \pmod d}
		\left| \sum_n \tau_k(n) \chi(n) w\left(\frac{n}{X}\right) \right|^2
		= 
		\sum_{\substack{d=qr\\ q>1}}\
		\displaystyle\sideset{}{^*}\sum_{\chi_1 (\bmod q)}
		\left| \sum_{(n,r)=1} \tau_k(n) \chi_1(n) w\left(\frac{n}{X}\right) \right|^2.
	\end{equation}
\end{lem}
\begin{proof}
	If $\chi (\bmod\ d)$ is non-principal, then there is a unique $q\mid d$ with $q>1$, and a unique primitive character $\chi_1 (\bmod\ q)$, such that, with $r=d/q$,
	\begin{equation} \label{eq:1124}
		\chi(n) = \chi_1(n) \chi_{0,r}(n),
	\end{equation}
	where $\chi_{0,r}$ is the principal character modulo $r$. Thus, replacing the sum $\displaystyle\sideset{}{^\prime}\sum_{\chi \pmod d}$ by $\displaystyle\sum_{\substack{d=qr\\ q>1}}\
	\displaystyle\sideset{}{^*}\sum_{\chi_1 (\bmod q)}$ and substituting \eqref{eq:1124} into the left side of \eqref{eq:1125b} give the right side of \eqref{eq:1125b}.
\end{proof}

\begin{lem}
	For $(mn,q_1) = 1$, we have
	\begin{equation} \label{eq:orthogonality}
		\frac{1}{\varphi^*(q_1)}
		\displaystyle\sideset{}{^*}\sum_{\chi (\bmod q_1)}
		\chi(m) \overline{\chi}(n)
		= 
		\begin{cases}
			1, & \text{if } n=m,\\
			\dfrac{1}{\varphi^*(q_1)}
			\displaystyle\sum_{\substack{q_1 = q_2 r_2\\ r_2 \mid m - n}}
			\mu(q_2) \varphi(r_2), & \text{if } n \neq m.
		\end{cases}
	\end{equation}
\end{lem}
\begin{proof}
	See, e.g., \cite[Lemma 2.7, pg. 348]{ChandeeLi2014}.
\end{proof}

The next is the well-known functional equation for primitive Dirichlet $L$-functions.

\begin{thmx}
	Let $\chi (\bmod\ q)$ be a primitive character. Then, for $\sigma>1$, we have
	\begin{equation} \label{eq:FE}
		L^k(1-s, \chi)
		= \gamma^k(s,\chi)
		L^k(s, \overline{\chi}),
	\end{equation}
	where
	\begin{equation} \label{eq:1026}
		\gamma(s,\chi) = 
		\varepsilon
		\left(\frac{q}{\pi}\right)^{s-\frac{1}{2}}
		\frac{\Gamma\left(\frac{s+a}{2}\right)}{\Gamma\left(\frac{1-s+a}{2}\right)},\
		|\varepsilon| = 1,
	\end{equation}
	with
	\begin{equation}
		a = \begin{cases}
			1, &\text{if } \chi(-1) = -1,\\
			0, &\text{if } \chi(-1) = 1.
		\end{cases}
	\end{equation}
\end{thmx}
\begin{proof}
	See, e.g, \cite[Section 5.9]{IwaniecKowalski}.
\end{proof}

\begin{lem} \label{lemma:tau}
	We have, for any $k\ge 2$,
	\begin{equation} \label{eq:839}
		\tau_k(p^j)
		= \binom{k+j-1}{k-1}.
	\end{equation}
\end{lem}
\begin{proof}
	By the relation
	\begin{equation}
		\left( \frac{1}{1-z} \right)^k
		= \sum_{j=0}^\infty
		\binom{k+j-1}{j} z^{j},\
		(|z| < 1),
	\end{equation}
	and the Euler product
	\begin{equation}
		\zeta(s)
		= \prod_p
		\frac{1}{1-p^{-s}},\
		(\sigma >1),
	\end{equation}
	we have
	\begin{equation} \label{eq:319a}
		\zeta^k(s)
		=
		\prod_p
		\left(
		\frac{1}{1-p^{-s}}
		\right)^k
		= \prod_p 
		\sum_{j=0}^\infty
		\binom{k+j-1}{j} 
		p^{-js},\
		(\sigma >1).
	\end{equation}
	We also have
	\begin{equation} \label{eq:319b}
		\zeta^k(s)
		=\sum_{n=1}^\infty
		\frac{\tau_k(n)}{n^s}
		= \prod_p
		\sum_{j=0}^\infty
		\tau_k(p^j) p^{-js},\
		(\sigma >1).
	\end{equation}
	Thus, by \eqref{eq:319a} and \eqref{eq:319b},
	\begin{equation}
		\tau_k(p^j)
		= \binom{k+j-1}{k-1}.
	\end{equation}
\end{proof}

\begin{lem}
	We have
	\begin{equation} \label{eq:magic}
		\sum_{j=0}^\infty
		\frac{\tau_k(p^j)^2}{p^{js}}
		=
		\left(1-\frac{1}{p^s} \right)^{-(2k-1)}
		\sum_{j=0}^{k-1} \frac{\binom{k-1}{j}^2}{p^{js}}.
	\end{equation}
\end{lem}
\begin{proof}
	We have the identity
	\begin{equation} \label{eq:839b}
		\sum_{j=0}^\infty
		\frac{\binom{k+j-1}{k-1}^2}{p^{js}}
		=
		\left(1-\frac{1}{p^s} \right)^{-(2k-1)}
		\sum_{j=0}^{k-1} \frac{\binom{k-1}{j}^2}{p^{js}}.
	\end{equation}
	This identity \eqref{eq:839b} follows from a hypergeometric relation; see, e.g., \cite[Proposition C.2, pg. 1295]{KowalskiRicotta2014}. By \eqref{eq:839}, the left side of \eqref{eq:839b} is that of \eqref{eq:magic}.
\end{proof}

The next lemma is based on Parseval formula for Mellin transforms, see, e.g., \cite[Theorem 1.17, pg. 33]{Yakubovich1996}, and makes use of the condition \eqref{eq:weightwNormalization} in statement of Corollary \ref{cor:2}.

\begin{lem} \label{lemma:w}
	We have, with $f_{r, \chi_1}$ defined as in \eqref{eq:f_r} and $\chi_1$ a generic primitive character,
	\begin{equation} \label{eq:f_r(1)=1}
		\mathcal{M}[|f_{r, \chi_1}|^2](1)
		= 1.
	\end{equation}
\end{lem}
\begin{proof}
	By definition of the Mellin transform and Parseval formula for the Mellin transform (\cite[Theorem 1.17, pg. 33]{Yakubovich1996}), we have
	\begin{align}
		\mathcal{M}[|f_{r, \chi_1}|^2](1)
		&= \int_0^\infty
		f_{r, \chi_1}(y,\chi_1)
		\overline{f_{r, \chi_1}(y,\chi_1)}
		dy
		\\&=
		\frac{1}{2\pi i}
		\int_{(2)}
		\mathcal{M}[f_{r, \chi_1}](s)
		\mathcal{M}[\overline{f_{r, \chi_1}}](1-s)
		ds.
	\end{align}
	By definition of the Mellin transform again and by \eqref{eq:f_r}, the above is equal to
	\begin{align}
		\frac{1}{2\pi i}
		&\int_{(2)}
		\mathcal{M}[w](1-s)
		\frac{\Gamma^k\left(\frac{s+a}{2}\right)}{\Gamma^k\left(\frac{1-s+a}{2}\right)}
		\prod_{p \mid r}
		\left(
		\frac{1-\frac{\chi_1(p)}{p^{1-s}}}{1-\frac{\overline{\chi_1}(p)}{p^s}}
		\right)^k
		\\& \qquad \times
		\mathcal{M}[w](s)
		\frac{\Gamma^k\left(\frac{1-s+a}{2}\right)}{\Gamma^k\left(\frac{s+a}{2}\right)}
		\prod_{p \mid r}
		\left(
		\frac{1-\frac{\overline{\chi_1}(p)}{p^{s}}}{1-\frac{{\chi_1}(p)}{p^{1-s}}}
		\right)^k
		ds
		\\&=
		\frac{1}{2 \pi i}
		\int_{(2)}
		\mathcal{M}[w](1-s)
		\mathcal{M}[w](s) ds
		= 
		\int_0^\infty |w(x)|^2 dx,
	\end{align}
	where the last equality follows by Parseval formula for Mellin transforms once more time. By \eqref{eq:weightwNormalization}, this gives the right side of \eqref{eq:f_r(1)=1}.
\end{proof}

\section{Proof of Theorem \ref{theorem:1} \nameref{theorem:1}}
\label{sec:ProofOfTheorem}

\subsection{Step 1: Applying the functional equation and reduction in $q$} 
\label{subsection:step1}

We first rewrite the variance \eqref{eq:thm1} as a sum over primitive characters, in order to apply the functional equation. By Lemmas \ref{lem:non-principalCharacters} and \ref{lem:primitiveCharacters}, we have
\begin{align}
	\sum_{\substack{1\le a\le d\\ (a,d)=1}}
	\left|
	\Delta_w(\tau_k;X,d,a)
	\right|^2
	=
	\frac{1}{\varphi(d)}
	\sum_{\substack{d=qr\\ q>1}}\
	\displaystyle\sideset{}{^*}\sum_{\chi_1 (\bmod q)}
	\left| \sum_{(n,r)=1} \tau_k(n) \chi_1(n) w\left(\frac{n}{X}\right) \right|^2.
\end{align}
Substituting
\begin{equation}
	w\left(\frac{n}{X}\right)
	= \frac{1}{2 \pi i} \int_{(\sigma)}
	\mathcal{M}[w](s) \left(\frac{n}{X}\right)^{-s} ds,\ (\sigma >1),
\end{equation}
into the above and interchanging the order of summation and integration, the left side of \eqref{eq:thm1} can be written as
\begin{equation} \label{eq:1049}
	\frac{1}{\varphi(d)}
	\sum_{\substack{d=qr\\ q>1}}\
	\displaystyle\sideset{}{^*}\sum_{\chi_1 (\bmod q)}
	\left| \frac{1}{2\pi i} 
	\int_{(\sigma)}
	\mathcal{M}[w](s) X^s
	\sum_{(n,r)=1} \frac{\tau_k(n) \chi_1(n)}{n^s} ds \right|^2,\
	(\sigma >1).
\end{equation}
Since $q>1$, $\chi_1$ modulo $q$ is non-principal, thus, the function $\displaystyle\sum_{(n,r)=1} \frac{\tau_k(n) \chi_1(n)}{n^{s}} = L(s, \chi_1 \chi_{0,r})^k$ has no poles. Moreover, since zero is not in the  support of $w(y)$, the Mellin transform $\mathcal{M}[w](s)$ also has no poles. Hence, the integrand in \eqref{eq:1049} is entire and we may shift the line of integration to the left of the one-line. 

Shifting the line of integration to $\sigma=-\epsilon$, applying the functional equation \eqref{eq:FE}, and making a change of variables $s$ to $1-s$, \eqref{eq:1049} is equal to
\begin{equation} \label{eq:433}
	\frac{1}{\varphi(d)}
	\sum_{\substack{d=qr\\ q>1}}\
	\displaystyle\sideset{}{^*}\sum_{\chi_1 (\bmod q)}
	\left|
	\frac{1}{2\pi i} \int_{(1+\epsilon)}
	\mathcal{M}[w](1-s) X^{1-s} \gamma^k(s,\chi_1)
	L^k(s,\overline{\chi}_1)
	\prod_{p \mid r}
	\left(1-\frac{\chi_1(p)}{p^{1-s}}\right)^k ds
	\right|^2.
\end{equation}
Let
\begin{equation} \label{eq:Q}
	Q := d^{\frac{3c+2}{3k+2}}.
\end{equation}
This choice of $Q$ comes from equating the error terms from \eqref{eq:102} with \eqref{eq:245} below. We will now show that contributions from characters with modulus $q < Q$ are negligible. 

Suppose 
\begin{equation} \label{eq:q<Q}
	q < Q.
\end{equation}
By repeated integration by parts and rapid decay of derivatives of the smooth weight $w(y)$, we have
\begin{equation} \label{eq:what}
	\mathcal{M}[w](1-s) \ll \frac{1}{(1 + |t|)^A}
\end{equation}
for any $A>0$. Also, for $s=1+\epsilon + it$, we have, by \eqref{eq:q<Q}, the bounds
\begin{equation}
	\gamma^k(s,\chi_1)
	\ll q^{k\left(\frac{1}{2} + \epsilon\right)}
	< Q^{k\left(\frac{1}{2} + \epsilon\right)},
\end{equation}
\begin{equation}
	L^k(s,\overline{\chi}_1)
	\ll 1,
\end{equation}
\begin{equation}
	\prod_{p \mid r}
	\left(1-\frac{\chi_1(p)}{p^{1-s}}\right)^k
	\ll r^\epsilon.
\end{equation}
Thus, by the above estimates, the integral in \eqref{eq:433} is
\begin{equation}
	\ll Q^{\frac{k}{2} + \epsilon}.
\end{equation}
Hence, by the above,
\begin{align} \label{eq:400}
	&\frac{1}{\varphi(d)}
	\sum_{\substack{d=qr\\ 1< q < Q}}\
	\displaystyle\sideset{}{^*}\sum_{\chi_1 (\bmod q)}
	\left|
	\frac{1}{2\pi i} \int_{(1+\epsilon)}
	\mathcal{M}[w](1-s) X^{1-s} \gamma^k(s,\chi_1)
	\right.
	\\& \left. \quad \times
	L^k(s,\overline{\chi}_1)
	\prod_{p \mid r}
	\left(1-\frac{\chi_1(p)}{p^{1-s}}\right)^k ds
	\right|^2
	\ll
	Q^{k+\epsilon}
	\frac{1}{\varphi(d)}
	\sum_{\substack{d=qr\\ 1 < q < Q}}
	\varphi^*(q).
\end{align}
By \eqref{eq:q<Q}, \eqref{eq:tauBound}, the crude bounds $\varphi^*(q) \ll q$ and
\begin{equation} \label{eq:varphi}
	\frac{1}{\varphi(d)} \ll d^{-1+\epsilon},
\end{equation}
we have
\begin{equation}
	\frac{1}{\varphi(d)}
	\sum_{\substack{d=qr\\ 1 < q < Q}}
	\varphi^*(q)
	\ll d^{-1+\epsilon} Q.
\end{equation}
With this, \eqref{eq:400} is bounded by
\begin{equation} \label{eq:102}
	\ll Q^{k+1} d^{-1 + \epsilon}
	\ll X^{1+\epsilon} \frac{Q^{k+1}}{Xd}.
\end{equation}
By the choice \eqref{eq:Q}, the above is
\begin{equation}
	\ll X^{1+\epsilon} d^{\frac{c-k}{3k+2}}
	\ll X^{1+\epsilon} d^{-\frac{\delta}{3k+2}}
\end{equation}
uniformly for all $c$ in \eqref{eq:c}. It thus suffices to assume $q \ge Q$. 

\subsection{Step 2: Reduction to the dual sum}
\label{subsection:step2} 
Suppose $q \ge Q$. We now introduce the parameter $N$ and truncate the $n$ sum in $L^k(s,\overline{\chi_1})$. Terms with $n\le N$ account for the main term, and terms with $n>N$ will be shown in this step to be small. By Euler products, we have
\begin{equation}
	\sum_{\substack{(n,r)=1}}
	\frac{\tau_k(n) \overline{\chi_1}(n)}{n^s}
	=
	\prod_{p\mid r} \left(1-\frac{\overline{\chi}_1(p)}{p^s}\right)^{k}
	L^k(s,\overline{\chi_1}),\
	(\sigma >1).
\end{equation}
Thus, writing
\begin{equation} \label{eq:1202}
	L^k(s,\overline{\chi}_1)
	= 
	\prod_{p\mid r} \left(1-\frac{\overline{\chi}_1(p)}{p^s}\right)^{-k}
	\left( \sum_{\substack{n\le N\\ (n,r)=1}} \frac{\tau_k(n) \overline{\chi}_1(n)}{n^s}
	+ \sum_{\substack{n > N\\ (n,r)=1}} \frac{\tau_k(n) \overline{\chi}_1(n)}{n^s} \right),
\end{equation}
for $L^k(s,\overline{\chi}_1)$ in \eqref{eq:433}, with $N$ to be chosen below, we have, by \eqref{eq:1202}, \eqref{eq:433}, and \eqref{eq:1049}, 
\begin{align} \label{eq:157}
	&\sum_{\substack{1\le a\le d\\ (a,d)=1}}
	\left|
	\Delta_w(\tau_k;X,d,a)
	\right|^2
	\\& = 
	\frac{1}{\varphi(d)}
	\sum_{\substack{d=qr\\ q \ge Q}}\
	\displaystyle\sideset{}{^*}\sum_{\chi_1 (\bmod q)}
	\left|
	\frac{1}{2\pi i} \int_{(1+B)}
	\left(
	\sum_{\substack{n\le N\\ (n,r)=1}}
	+ 
	\sum_{\substack{n > N\\ (n,r)=1}}
	\right)
	\frac{\tau_k(n) \overline{\chi}_1(n)}{n^s}
	\right.
	\\& \left. \qquad
	\times \mathcal{M}[w](1-s) X^{1-s} \gamma^k(s,\chi_1)
	\prod_{p \mid r}
	\left(\frac{1-\frac{\chi_1(p)}{p^{1-s}}}{1-\frac{\overline{\chi}_1(p)}{p^s}}\right)^k ds
	\right|^2
	+ O \left(
	X^{1+\epsilon} d^{-\frac{\delta}{3k+2}}
	\right),
\end{align}
where we have moved the line of integration to $\sigma = 1 +B$, $B>0$. Analogously to \textbf{Step 1}, we will show contributions from terms with $n>N$ in the above are negligible. The main difference in this step is that there is an additional saving from the sum $\sum_{n>N} n^{-\sigma + \epsilon}$. We have, for $s=1+B+ it$,
\begin{equation}
	\sum_{\substack{n>N\\ (n,r)=1}}
	\frac{\tau_k(n) \overline{\chi_1}(n)}{n^s}
	\ll N^{-B + \epsilon},
\end{equation}
\begin{equation}
	X^{1-s} \ll X^{-B},
\end{equation}
\begin{equation}
	\gamma^k(s,\chi_1)
	\ll q^{k\left(B + \frac{1}{2} \right)},
\end{equation}
\begin{equation}
	L^k(s,\overline{\chi}_1)
	\ll 1,
\end{equation}
\begin{equation}
	\prod_{p \mid r}
	\left(1-\frac{\chi_1(p)}{p^{1-s}}\right)^k
	\left(1-\frac{\overline{\chi}_1(p)}{p^s}\right)^{-k}
	\ll r^{kB + \epsilon}.
\end{equation}
Thus, by the above estimates, the contribution from terms with $n>N$ in the integral on the right side of \eqref{eq:157} is
\begin{equation} \label{eq:209}
	\ll N^{-B + \epsilon}
	X^{-B}
	q^{k\left(B+\frac{1}{2}\right)} r^{kB+ \epsilon}.
\end{equation}
We choose
\begin{equation} \label{eq:N}
	N = d^{k-c+\frac{\delta}{2}}.
\end{equation}
and
\begin{equation} \label{eq:B}
	B = \frac{A+k}{\delta}.
\end{equation} 
With the choices \eqref{eq:N} and \eqref{eq:B}, the estimate \eqref{eq:209} is
\begin{equation}
	\ll d^{- \frac{A}{2} + \epsilon},
\end{equation}
for any $A>0$. This bound shows that the contribution from terms with $n>N$ on the right side of \eqref{eq:157} can be absorbed into the big-oh term.

We are now reduced to terms with $q \ge Q$ and $n\le N$. To summarize, we now have
\begin{align} \label{eq:814}
	\sum_{\substack{1\le a\le d\\ (a,d)=1}}
	&\left|
	\Delta_w(\tau_k;X,d,a)
	\right|^2
	\\&=
	\frac{1}{\varphi(d)}
	\sum_{\substack{d=qr\\ q \ge Q}}\
	\displaystyle\sideset{}{^*}\sum_{\chi_1 (\bmod q)}
	\left|
	\sum_{\substack{n\le N\\ (n,r)=1}}
	\tau_k(n) \overline{\chi}_1(n)
	\frac{1}{2\pi i} 
	\int_{(1+B)}
	\mathcal{M}[w](1-s) \frac{X}{(nX)^s}
	\right.
	\\& \left. \qquad \times
	\gamma^k(s,\chi_1)
	\prod_{p \mid r}
	\left(\frac{1-\frac{\chi_1(p)}{p^{1-s}}}{1-\frac{\overline{\chi}_1(p)}{p^s}}\right)^k ds
	\right|^2
	+ O\left(
	X^{1+\epsilon} d^{-\frac{\delta}{3k+2}}
	\right).
\end{align}
For the evaluation of the main term, we will move the line of integration to the critical line $s = 1/2 + it$. 

\subsection{Step 3: Off-diagonal analysis} 
\label{subsection:step3}
Let $s=1/2 + it$. We expand out the square on the right side of \eqref{eq:814} and first treat the off-diagonal terms. Let
\begin{equation} \label{eq:Y}
	XY = d^k.
\end{equation}
We write
\begin{equation}
	\frac{X}{(nX)^s} \gamma^k(s,\chi_1)
	= \varepsilon^k \frac{X^{1/2} r^{k/2} \pi^{k/2}}{Y^{1/2}}
	\left(\frac{n r^k \pi^k}{Y}\right)^{-s}
	\frac{\Gamma\left(\frac{s+a}{2}\right)^k}{\Gamma\left(\frac{1-s+a}{2}\right)^k},
\end{equation}
and
\begin{align}
	\prod_{p \mid r}
	\left(\frac{1-\frac{\chi_1(p)}{p^{1-s}}}{1-\frac{\overline{\chi}_1(p)}{p^s}}\right)^k
	&=
	\sum_{i_1, \cdots, i_\ell =0}^k\
	\sum_{j_1, \cdots, j_\ell =0}^\infty
	\binom{k}{i_1} \cdots \binom{k}{i_\ell}
	\binom{-k}{j_1} \cdots \binom{-k}{j_\ell}
	\\& \quad \times 
	(-1)^{i_1 + \cdots + i_\ell}
	\frac{\chi_1(p_1^{i_1} \cdots p_\ell^{i_\ell}) \overline{\chi}_1(p_1^{j_1} \cdots p_\ell^{j_\ell})}{p_1^{i_1 (1-s)} \cdots p_\ell^{i_\ell (1-s)} p_1^{j_1 s} \cdots p_\ell^{j_\ell s}},
\end{align}
where
\begin{equation} \label{eq:1204}
	r = p_1^{\alpha_1} p_2^{\alpha_2} \cdots p_\ell^{\alpha_\ell}
\end{equation}
is the prime factorization of $r$.
Denote
\begin{align}
	f_r(y)
	&= 
	\sum_{i_1, \cdots, i_\ell =0}^k\
	\sum_{j_1, \cdots, j_\ell =0}^\infty
	\binom{k}{i_1} \cdots \binom{k}{i_\ell}
	\binom{-k}{j_1} \cdots \binom{-k}{j_\ell}
	(-1)^{i_1 + \cdots + i_\ell}
	\\& \qquad \times
	\frac{1}{2\pi i}
	\int_{(\frac{1}{2})}
	\mathcal{M}[w](1-s) y^{-s}
	\frac{\Gamma\left(\frac{s+a}{2}\right)^k}{\Gamma\left(\frac{1-s+a}{2}\right)^k}
	\frac{p_1^{i_1 (s-1)} \cdots p_\ell^{i_\ell (s-1)}}{p_1^{j_1 s} \cdots p_\ell^{j_\ell s}}
	ds.
\end{align}
We write the off-diagonal terms $n\neq m$ on the right side of \eqref{eq:814} as
\begin{align} \label{eq:849}
	\quad \quad \quad
	\frac{X}{Y}
	&\frac{1}{\varphi(d)}
	\sum_{\substack{d=qr\\ q \ge Q}}\
	(\pi r)^k
	\sum_{\substack{n \le N\\ (n,r) = 1}}
	\tau_k(n)
	f_r\left(\frac{n r^k \pi^k}{Y}\right)
	\sum_{\substack{m\le N\\ (m,r)=1\\ m\neq n}}
	\tau_k(m)
	\overline{f_r\left(\frac{m r^k \pi^k}{Y} \right)}
	\\& \times
	\displaystyle\sideset{}{^*}\sum_{\chi_1 (\bmod q)}	
	\overline{\chi}_1(n p_1^{j_1} \cdots p_\ell^{j_\ell})
	\chi_1(m p_1^{i_1} \cdots p_\ell^{i_\ell}).
\end{align}
Since $\chi_1$ and $\overline{\chi}_1$ are characters modulo $q$, we have $\chi_1(m p_1^{i_1} \cdots p_\ell^{i_\ell}) = \overline{\chi}_1(n p_1^{j_1} \cdots p_\ell^{j_\ell}) = 0$, unless
\begin{equation} \label{eq:1122}
	(n p_1^{j_1} \cdots p_\ell^{j_\ell},q) = (m p_1^{i_1} \cdots p_\ell^{i_\ell}, q) = 1.
\end{equation}
Also, note that, since $n\neq m$, $p_i \mid r$ and $(nm,r)=1$,
\begin{equation} \label{eq:1122a}
	n p_1^{j_1} \cdots p_\ell^{j_\ell} 
	\neq
	m p_1^{i_1} \cdots p_\ell^{i_\ell}.
\end{equation}
Thus, by \eqref{eq:1122}, \eqref{eq:1122a}, and \eqref{eq:orthogonality}, we have
\begin{equation}
	\displaystyle\sideset{}{^*}\sum_{\chi_1 (\bmod q)}	
	\overline{\chi}_1(n p_1^{j_1} \cdots p_\ell^{j_\ell})
	\chi_1(m p_1^{i_1} \cdots p_\ell^{i_\ell})
	= \sum_{\mathfrak{d} \mid (q, n p_1^{j_1} \cdots p_\ell^{j_\ell} - m p_1^{i_1} \cdots p_\ell^{i_\ell} ) }
	\varphi(\mathfrak{d}) \mu \left(\frac{q}{\mathfrak{d}}\right).
\end{equation}
Let
\begin{equation}
	b_r \equiv p_1^{j_1} \cdots p_\ell^{j_\ell} \overline{p_1^{i_1} \cdots p_\ell^{i_\ell}} (\bmod \mathfrak{d}),
\end{equation}
where the bar denotes inverse modulo $\mathfrak{d}$. Note that, by \eqref{eq:1122a}, 
\begin{equation} \label{eq:1122b}
	m \neq n b_r.
\end{equation}
Writing the condition $\mathfrak{d} \mid n p_1^{j_1} \cdots p_\ell^{j_\ell} - m p_1^{i_1} \cdots p_\ell^{i_\ell}$ as $m \equiv  n b_r (\bmod\ \mathfrak{d})$, \eqref{eq:849} can be rewritten as
\begin{align} \label{eq:900}
	\frac{X}{Y}
	&\frac{1}{\varphi(d)}
	\sum_{\substack{d=qr\\ q \ge Q}}\
	\pi^k r^k
	\sum_{\mathfrak{d} \mid q} 
	\varphi(\mathfrak{d}) \mu\left(\frac{q}{\mathfrak{d}} \right)
	\sum_{\substack{n \le N\\ (n,r) = 1}}
	\tau_k(n)
	f_r\left(\frac{n r^k \pi^k}{Y}\right)
	\\ & \times
	\sum_{\substack{m\le N\\ 
			(m,r)=1\\ 
			m\neq n\\ 
			m \equiv  n b_r (\bmod \mathfrak{d}) }}
	\tau_k(m)
	\overline{f_r\left(\frac{m r^k \pi^k}{Y} \right)}.
\end{align}
If $\mathfrak{d} > N$, then the $m$ sum in the above is empty, by \eqref{eq:1122b}. If $\mathfrak{d} \le N$, then the $m$ sum in \eqref{eq:900} is bounded by
\begin{align}
	r^\epsilon Y^\epsilon N^\epsilon
	\sum_{\substack{m\le N\\ m \equiv nb_r (\bmod \mathfrak{d})}}
	\frac{Y^{1/2}}{m^{1/2} r^{k/2}}
	\ll (rYN)^\epsilon 
	\frac{Y^{1/2}}{r^{k/2}} \frac{N^{1/2}}{\mathfrak{d}}.
\end{align}
With this estimate, contribution from $\mathfrak{d} \le N$ to \eqref{eq:900} is bounded by
\begin{align} \label{eq:off-diag}
	\frac{X^{1+\epsilon}}{Y}
	\frac{1}{\varphi(d)}
	\sum_{\substack{d=qr\\ q \ge Q}}\
	r^k
	\sum_{ \mathfrak{d} \mid q} 
	\frac{\varphi(\mathfrak{d})}{\mathfrak{d}}
	\frac{YN}{r^k}
	\ll X^{1+\epsilon} \frac{N}{\varphi(d)}
	\ll X^{1+\epsilon} d^{-\frac{\delta}{2}},
\end{align}
by \eqref{eq:N} and \eqref{eq:varphi}. This estimate for the off-diagonal terms on the right side of \eqref{eq:814} with $\mathfrak{d} \le N$ is acceptable for \eqref{eq:thm1}. 

We last treat the diagonal terms in \eqref{eq:814}.

\subsection{Step 4: Diagonal analysis} 
\label{subsection:step4}
For $\sigma>1$, let
\begin{equation} \label{eq:f_r}
	f_{r, \chi_1}(y)
	= \frac{1}{2 \pi i} \int_{(\sigma)}
	\mathcal{M}[w](1-s) y^{-s}
	\frac{\Gamma\left(\frac{s+a}{2}\right)^k}{\Gamma\left(\frac{1-s+a}{2}\right)^k}
	\prod_{p \mid r}
	\left(\frac{1-\frac{\chi_1(p)}{p^{1-s}}}{1-\frac{\overline{\chi}_1(p)}{p^s}}\right)^k
	ds.
\end{equation}
This function can also be written as an inverse Mellin transform as
\begin{equation} \label{eq:128}
	f_{r, \chi_1}(y)
	= \mathcal{M}^{-1}
	\left[
	s \mapsto
	\mathcal{M}[w](1-s)
	\frac{\Gamma\left(\frac{s+a}{2}\right)^k}{\Gamma\left(\frac{1-s+a}{2}\right)^k}
	\prod_{p \mid r}
	\left(\frac{1-\frac{\chi_1(p)}{p^{1-s}}}{1-\frac{\overline{\chi}_1(p)}{p^s}}\right)^k
	\right](y).
\end{equation}
We thus have, by \eqref{eq:814}, \eqref{eq:off-diag}, and \eqref{eq:f_r},
\begin{align} \label{eq:star2}
	\sum_{\substack{1\le a\le d\\ (a,d)=1}}
	&\left|
	\Delta_w(\tau_k;X,d,a)
	\right|^2
	\\&=
	X^2
	\frac{1}{\varphi(d)}
	\sum_{\substack{d=qr\\ q \ge Q}}
	\frac{\pi^k}{q^k}
	\displaystyle\sideset{}{^*}\sum_{\chi_1 (\bmod q)}
	\sum_{\substack{n\le N\\ (n,r)=1}}
	\tau_k(n)^2 \left| f_{r, \chi_1}\left(\frac{n X \pi^k}{q^k}\right)\right|^2
	\\&+ O\left(
	X^{1+\epsilon} d^{-\frac{\delta}{3k+2}}
	\right).
\end{align}
We also have
\begin{equation} \label{eq:920}
	\left|f_{r, \chi_1}\left(\frac{n X \pi^k}{q^k}\right) \right|^2
	\ll \left(\frac{q^k}{nX}\right)^{2B + 2} r^{2kB +\epsilon},
\end{equation}
for any $B>0$. We now write
\begin{align}
	\sum_{\substack{n\le N\\ (n,r)=1}}
	&\tau_k(n)^2 \left| f_{r, \chi_1}\left(\frac{n X \pi^k}{q^k}\right)\right|^2
	\\&= \left(\sum_{\substack{n=1\\ (n,r)=1}}^\infty - \sum_{\substack{n>N\\ (n,r)=1}} \right)
	\tau_k(n)^2 \left| f_{r, \chi_1}\left(\frac{n X \pi^k}{q^k}\right)\right|^2.
\end{align}
For terms $n>N$ in the above, we have, by \eqref{eq:920},
\begin{align}
	X^2
	&\frac{1}{\varphi(d)}
	\sum_{\substack{d=qr\\ q \ge Q}}
	\frac{\pi^k}{q^k}
	\displaystyle\sideset{}{^*}\sum_{\chi_1 (\bmod q)}
	\sum_{\substack{n>N\\ (n,r)=1}}
	\tau_k(n)^2 \left| f_{r, \chi_1}\left(\frac{n X \pi^k}{q^k}\right)\right|^2
	\\& \label{eq:starsquared}
	\ll d^{2kB+k} N^{-(2B+1)} X^{-2B+\epsilon}
	\ll X^{1+\epsilon} d^{-\delta(B+1/2)},
\end{align}
for any $B>0$. This amount is admissible for \eqref{eq:thm1}.

Finally, we have, by inverse Mellin transform of $f_r$,
\begin{align} \label{eq:729}
	\sum_{\substack{n=1\\ (n,r)=1}}^\infty
	\tau_k(n)^2 
	&\left| f_{r, \chi_1}\left(\frac{n X \pi^k}{q^k}\right)\right|^2
	\\&
	= \frac{1}{2\pi i}
	\int_{(2)}
	\mathcal{M}[|f_{r, \chi_1}|^2](s) \left(\frac{q^k}{X \pi^k} \right)^s
	\sum_{\substack{n=1\\ (n,r)=1}}^\infty \frac{\tau_k(n)^2}{n^s} ds,
\end{align}
for $\sigma>1$. Therefore, by \eqref{eq:star2}, \eqref{eq:starsquared}, and the above, we obtain
\begin{align} \label{eq:203}
	&\quad\quad \sum_{\substack{1\le a\le d\\ (a,d)=1}}
	\left|
	\Delta_w(\tau_k;X,d,a)
	\right|^2
	\\&=
	\frac{X}{\varphi(d)}
	\sum_{\substack{d=qr\\ q \ge Q}}
	\displaystyle\sideset{}{^*}\sum_{\chi_1 (\bmod q)}
	\frac{1}{2\pi i}
	\int_{(2)}
	\mathcal{M}[|f_{r, \chi_1}|^2](s) \left(\frac{q^k}{X \pi^k} \right)^{s-1}
	\sum_{\substack{n=1\\ (n,r)=1}}^\infty \frac{\tau_k(n)^2}{n^s} ds
	\\&+ O\left(
	X^{1+\epsilon} d^{-\frac{\delta}{3k+2}}
	\right).
\end{align}
This proves \eqref{eq:thm1}. \qed

\section{Proof of Corollary \ref{cor:2} \nameref{cor:2}}
\label{sec:ProofofCorollary}

\subsection{Step a: Full main term as a residue plus an error term} 
\label{subsection:stepa}
We start with \eqref{eq:203}. The integrand in \eqref{eq:203} has a pole of order $k^2$ at $s=1$. The residue of this pole is equal to
\begin{equation} \label{eq:819}
	\mathcal{M}[|f_{r, \chi_1}|^2](1)
	P_{k^2-1} \left( \log \left( \frac{(q/ \pi)^k}{X} \right), r \right)
	= P_{k^2-1} \left( \log \left( \frac{(q/ \pi)^k}{X} \right), r \right),
\end{equation}
by \eqref{eq:f_r(1)=1}, where
\begin{equation} \label{eq:745}
	P_{k^2-1}(y,r) 
	= \sum_{j=0}^{k^2-1}
	b_{j}(r) y^j
\end{equation}
is an explicit polynomial in $y$ of degree $k^2-1$. The leading coefficient $b_{k^2-1}(r)$ in \eqref{eq:745} is equal to
\begin{equation}
	\frac{a_k(r)}{(k^2-1)!}
	= \frac{1}{(k^2-1)!}
	\lim_{s\to 1^+}
		(s-1)^{k^2}
		\sum_{\substack{n=1\\ (n,r)=1}}^\infty
		\frac{\tau_k(n)^2}{n^s},
\end{equation}
by \eqref{eq:akd1}. Note that 
\begin{equation}
	a_k(1)
	= \lim_{s\to 1^+}
	(s-1)^{k^2}
	\sum_{\substack{n=1}}^\infty
	\frac{\tau_k(n)^2}{n^s}
	= a_k,
\end{equation}
and that the factor $a_k(r)$ is uniformly bounded in $r$, since we can rewrite
\begin{equation}
	\sum_{\substack{n=1\\ (n,r)=1}}^\infty
	\frac{\tau_k(n)^2}{n^s}
	= 
	\sum_{\substack{n=1}}^\infty
	\frac{\tau_k(n)^2}{n^s}
	\prod_{p \mid r}
	\left(
	\sum_{j=0}^\infty \frac{\tau_k(p^j)^2}{p^{js}}
	\right),
\end{equation}
so
\begin{equation} \label{eq:akr}
	a_k(r)
	= a_k 
	\prod_{p \mid r}
	\left(
	\sum_{j=0}^\infty \frac{\tau_k(p^j)^2}{p^{j}}
	\right).
\end{equation}

We thus move the line of integration of the integral in \eqref{eq:203} to $\sigma=1/2$, passing the pole at $s=1$, pick up this residue and obtain
\begin{equation} \label{eq:159}
	P_{k^2-1} \left( \log \left( \frac{(q/ \pi)^k}{X} \right), r \right)
	+ O\left( \left( \frac{q^k}{X} \right)^{-\frac{1}{2}} \right)
\end{equation}
for this integral. Thus, we have, from \eqref{eq:203} and \eqref{eq:159},
\begin{align} \label{eq:thm1c}
	&\sum_{\substack{1\le a\le d\\ (a,d)=1}}
	\left|
	\Delta_w(\tau_k;X,d,a)
	\right|^2
	\\&=
	X
	\frac{1}{\varphi(d)}
	\sum_{\substack{d=qr\\ q \ge Q}}
	\varphi^*(q)
	\left[
	P_{k^2-1} \left( \log \left( \frac{(q/ \pi)^k}{X} \right), r \right)
	+ O\left( \left( \frac{q^k}{X} \right)^{-\frac{1}{2}} \right)
	\right]
	\\&+ O\left(
	X^{1+\epsilon} d^{-\frac{\delta}{3k+2}}
	\right).
\end{align}

\subsection{Step b: The error term and $\gamma_k(c)$} 
\label{subsection:stepb}
The first big-oh term in \eqref{eq:thm1c} contributes to the variance an amount
\begin{equation} \label{eq:245}
	\ll X
	\sum_{\substack{d=qr\\ q \ge Q}}
	\frac{\varphi^*(q)}{\varphi(d)}
	\left(\frac{q^k}{ X}\right)^{-\frac{1}{2}}
	\ll
	X^{1 + \epsilon}
	\frac{X^{\frac{1}{2}}}{Q^{\frac{k}{2}}}.
\end{equation}
By the choice \eqref{eq:Q}, this error term is equal to the second big-oh term in \eqref{eq:thm1c}, and so the two error terms can be absorb into one.

We now write
\begin{equation} \label{eq:141}
	\frac{(q/ \pi)^k}{X}
	= \frac{d^{k-c}}{(\pi r)^k}.
\end{equation}
With this, the leading order term in \eqref{eq:159} is equal to
\begin{align}
	\frac{a_k(r)}{(k^2-1)!} &\log \left( \frac{(q/ \pi)^k}{X} \right)^{k^2-1}
	\\&
	= \frac{a_k(r)}{(k^2-1)!} \left( (k-c)\log d - k \log \pi r \right)^{k^2-1}
	\\&= 
	\frac{a_k(r)}{(k^2-1)!} (k-c)^{k^2-1} (\log d)^{k^2-1}
	\\&\quad - \frac{a_k(r)}{(k^2-1)!} \binom{k^2-1}{1}
	(k-c)^{k^2-2} (\log d)^{k^2-2}
	k \log \pi r
	\\& 
	\quad\quad + \cdots 
	+ \frac{a_k(r)}{(k^2-1)!}
	(-k)^{k^2-1} (\log \pi r)^{k^2-1}
	\\&
	= a_k(r) \gamma_k(c) (\log d)^{k^2-1}
	\\&\quad - \frac{a_k(r)}{(k^2-1)!} \binom{k^2-1}{1}
	(k-c)^{k^2-2} (\log d)^{k^2-2}
	k \log \pi r
	\\&  \label{eq:135}
	\quad\quad + \cdots 
	+ \frac{a_k(r)}{(k^2-1)!}
	(-k)^{k^2-1} (\log \pi r)^{k^2-1},
\end{align}
where the last equality comes from \eqref{eq:gammakcsimple}.

Now that we have the factor $\gamma_k(c)$ showing up in the leading order term, we next evaluate the matching arithmetic constant $a_k(d)$ by showing that
\begin{equation} \label{eq:918}
	\frac{1}{\varphi(d)}
	\sum_{\substack{d=qr}}
	\varphi^*(q)
	a_k(r)
	= a_k(d),
\end{equation}
then estimate contributions from lower order main terms.

\subsection{Step c: Evaluating the matching arithmetic constant $a_k(d)$} 
\label{subsection:stepd}
To help with the analysis in this step, we add back terms with $q<Q$. This adds an amount
\begin{equation}
	X
	\sum_{\substack{d=qr\\ q < Q}}
	\frac{\varphi^*(q)}{\varphi(d)}
	a_k(r)
	\ll
	X Q d^{1-\epsilon}
	\ll X^{1-\frac{3 \delta}{3k+2} + \epsilon}.
\end{equation}
which is acceptable for the asymptotic. We rewrite the left side of \eqref{eq:918} as
\begin{equation} \label{eq:12:05}
	\frac{1}{\varphi(d)}
	(\mu\star \varphi \star a_k)(d).
\end{equation}
We first compute the Dirichlet series for $a_k(r)/a_k$. By \eqref{eq:akr}, the coefficients
\begin{equation}
	\frac{a_k(r)}{a_k}
	=
	\prod_{p \mid r}
	\left(
	\sum_{j=0}^\infty \frac{\tau_k(p^j)^2}{p^{j}}
	\right)
\end{equation}
are multiplicative in $r$. Thus, by Euler products, we have, by \eqref{eq:magic},
	\begin{align} 
		\sum_{r=1}^\infty
		\frac{a_k(r)/ a_k}{r^{s}}
		&
		= \prod_p 
		\sum_{j=0}^\infty
		\left(
		\sum_{i=0}^\infty
		\frac{\tau_k(p^i)^2}{p^{i}}\right)
		p^{-js}
		\\&
		=
		\prod_p \frac{(1-\frac{1}{p})^{2k-1}}{ \sum_{i=0}^{k-1} \binom{k-1}{i}^2 p^{-i}}
		\sum_{j=0}^\infty p^{-js}
		\\& \label{eq:1154a}
		=
		\zeta(s)
		\prod_p \frac{(1-\frac{1}{p})^{2k-1}}{ \sum_{i=0}^{k-1} \binom{k-1}{i}^2 p^{-i}}.
	\end{align}
	Next, we have the well known series
	\begin{equation} \label{eq:1154b}
		\sum_{n=1}^\infty \frac{\mu(n)}{n^s}
		= \frac{1}{\zeta(s)}
	\end{equation}
	and Euler product
	\begin{equation} \label{eq:1154c}
		\sum_{n=1}^\infty \frac{\varphi(n)}{n^s}
		= \prod_p
		\sum_{j=0}^\infty
		\frac{\varphi(p^j)}{p^{js}}.
	\end{equation}
Hence, by \eqref{eq:1154a}, \eqref{eq:1154b}, and \eqref{eq:1154c}, we have
\begin{align}
	\frac{1}{a_k}
	\sum_{d=1}^\infty
	\frac{(\mu\star \varphi \star a_k)(d)}{d^s}
	&= \frac{1}{\zeta(s)}
	\left(\prod_p
	\sum_{j=0}^\infty
	\frac{\varphi(p^j)}{p^{js}} \right)
	\left(
	\zeta(s)
	\prod_p \frac{(1-\frac{1}{p})^{2k-1}}{ \sum_{i=0}^{k-1} \binom{k-1}{i}^2 p^{-i}}
	\right)
	\\& =
	\prod_p
	\sum_{j=0}^\infty \frac{\varphi(p^j)(1-\frac{1}{p})^{2k-1}}{\sum_{i=0}^{k-1} \binom{k-1}{i}^2 p^{-i}}
	\frac{1}{p^{js}}.
\end{align}
Thus, at prime powers,
\begin{equation} \label{eq:315}
	\frac{1}{a_k}
	(\mu\star \varphi \star a_k)(p^j)
	= \frac{\varphi(p^j)(1-\frac{1}{p})^{2k-1}}{\sum_{i=0}^{k-1} \binom{k-1}{i}^2 p^{-i}}.
\end{equation}
Hence, if $d$ is factored as
\begin{equation} \label{eq:dPrimeFactor}
	d= p_1^{\alpha_1} \cdots
	p_v^{\alpha_v},
\end{equation}
with $p_j$ distinct, then \eqref{eq:315} implies that
\begin{align} \label{eq:1206}
	\frac{1}{a_k}
	(\mu\star \varphi \star a_k)(d)
	&= 
	\frac{1}{a_k}(\mu\star \varphi \star a_k)(p_1^{\alpha_1}) \cdots  
	\frac{1}{a_k}(\mu\star \varphi \star a_k)(p_v^{\alpha_v})
	\\&= \varphi(d) 
	\prod_{p \mid d}
	\frac{(1-\frac{1}{p})^{2k-1}}{\sum_{i=0}^{k-1} \binom{k-1}{i}^2 p^{-i}}
	\\&= \varphi(d) \frac{a_k(d)}{a_k}.
\end{align}
Thus, by the above, we have that
\begin{equation}
	(\mu\star \varphi \star a_k)(d)
	= \varphi(d) a_k(d).
\end{equation}
This proves \eqref{eq:918}. We last show
\begin{equation} \label{eq:829}
	\sum_{\substack{d=qr}}
	\frac{\varphi^*(q)}{\varphi(d)}
	a_k(r) \log r
	= o(\log d).
\end{equation}
With this, together with \eqref{eq:thm1c}, \eqref{eq:819}, \eqref{eq:141}, \eqref{eq:135}, \eqref{eq:918}, and induction, we will have, as $d\to \infty$,
\begin{align} \label{eq:836b}
	\sum_{\substack{1\le a\le d\\ (a,d)=1}}
	&\left|
	\Delta_w(\tau_k;X,d,a)
	\right|^2
	\\&
	\sim
	X
	\frac{1}{\varphi(d)}
	\sum_{\substack{d=qr}}
	\varphi^*(q)
	P_{k^2-1} \left( \log \left( \frac{(q/ \pi)^k}{X} \right), r \right)
	\\&
	\sim
	a_k(d) \gamma_k(c) X (\log d)^{k^2-1}.
\end{align}
We now prove \eqref{eq:829}. 

\subsection{Step d: Bounding lower order main terms} 
\label{subsection:stepc}
Similar to the previous step, contributions from $q<Q$ to \eqref{eq:829} 
\begin{equation} \label{eq:829b}
	\sum_{\substack{d=qr\\ q<Q}}
	\frac{\varphi^*(q)}{\varphi(d)}
	a_k(r) \log r
	\ll Q d^{-1+\epsilon}
	\ll d^{-\frac{3\delta}{3k+2} + \epsilon}
\end{equation}
is negligible. 

We have, for $\Re(s) >1$,
\begin{align} \label{eq:212}
	\frac{1}{a_k}
	\sum_{r=1}^\infty
	\frac{a_k(r) \log r}{r^s}
	= - \left( \sum_{r=1}^\infty \frac{a_k(r)/a_k}{r^s} \right)^\prime
	= 
	- C_k
	\zeta^\prime(s).
\end{align}
by \eqref{eq:1154a}, where $C_k=\prod_p \frac{(1-\frac{1}{p})^{2k-1}}{ \sum_{i=0}^{k-1} \binom{k-1}{i}^2 p^{-i}}$. We also have the well-known series
\begin{equation} \label{eq:von}
	- \frac{\zeta^\prime}{\zeta}(s)
	= \sum_{n=1}^\infty \frac{\Lambda(n)}{n^s},\ (\Re(s) > 1).
\end{equation}
Writing
\begin{equation}
	\varphi^* = \mu \star \varphi,
\end{equation}
we have, by \eqref{eq:212},
\begin{equation}
	\frac{1}{a_k}
	\sum_{d=1}^\infty
	\frac{(\mu \star \varphi \star a_k \log)(d)}{d^s}
	=
	C_k
	\sum_{n=1}^\infty \frac{\varphi(n)}{n^s}
	\left(-
	\frac{\zeta^\prime (s)}{\zeta(s)}\right).
\end{equation}
Thus, by \eqref{eq:von} and by equating coefficients from the above, we have that
\begin{equation} \label{eq:841}
	\sum_{\substack{d=qr}}
	\varphi^*(q)
	a_k(r) \log r
	= a_k C_k \sum_{d=qr} \varphi(q) \Lambda(r).
\end{equation}
The $r$ sum on the right of \eqref{eq:841} is restricted to prime powers of the form
\begin{equation} \label{eq:1154}
	r = p_j^{\beta_j},\
	1\le j\le v,\
	1\le \beta_j \le \alpha_j.
\end{equation}
For each $r$ in \eqref{eq:1154}, we have
\begin{equation}
	\varphi(q) \le q = \frac{d}{r}
	\le \frac{d}{p_j},
\end{equation}
since $\beta_j \ge 1$. Thus, summing over primes dividing $d$, we have, by the above,
\begin{equation} \label{eq:1223a}
	\sum_{d=qr} \varphi(q) \Lambda(r)
	\ll 
	d \sum_{p \mid d} \frac{\log p}{p}
\end{equation}
Introducing
\begin{equation} \label{eq:P}
	P = (\log d)^2,
\end{equation}
the above is bounded by
\begin{align}
	d \left( \sum_{\substack{p\mid d \\ p \le P}} \frac{\log p}{p} 
	+ \sum_{\substack{p\mid d \\ p > P}} \frac{\log p}{p}  \right)
	\ll d \left( \log P \sum_{\substack{p\mid d \\ p \le P}} \frac{1}{p} 
	+ \frac{1}{P} \sum_{\substack{p\mid d \\ p >P}} \log p \right)
	\ll d (\log P)^2
\end{align}
Thus, by the above, \eqref{eq:1223a}, \eqref{eq:841}, and the bound $\varphi(d) \gg d/ \log \log d$, the left side of \eqref{eq:829} is bounded by
\begin{align}
	(\log \log d)^3 = o(\log d).
\end{align}
This proves \eqref{eq:829}, and, therefore, the asymptotic \eqref{eq:cor:2}, as $d$ goes to infinity. \qed

\bigskip
\noindent
\emph{Acknowledgments}.
The author is indebted to Bradley Rodgers for suggesting this problem and for extensive discussions around this topic. He also benefited from the many discussions with J. Brian Conrey and is grateful for all his support and generous accommodation. Special thanks to Zeev Rudnick for suggesting to formulate Conjecture \ref{conj:1} and Stephen Lester for useful comments to an earlier draft. The author is also grateful for the referee for through-all reading, pointing out a flaw in the original choice of the \hyperlink{parameter:N}{parameter $N$} and a notational issue, and for making countlessly helpful suggestions. This research is supported in part by the National Science Foundation (NSF) grant Focus Research Group DMS-1854398, the American Institute of Mathematics, and Queen's University.

\end{document}